\documentclass[a4paper,oneside]{amsart}

\usepackage{geometry}               
\usepackage[english]{babel}
\usepackage[latin1]{inputenc}
\usepackage{amssymb,amsmath,amsthm,amsfonts}
\usepackage{graphicx}
\usepackage[usenames,dvipsnames]{xcolor}
\usepackage{mathtools}
\usepackage{cancel}
\mathtoolsset{showonlyrefs}
\usepackage[colorlinks]{hyperref}
\usepackage{filecontents}
\usepackage[normalem]{ulem}

\usepackage{tikz}
\usepackage{pgfplots}

\pgfplotsset{compat = newest}

\newcommand{\Z}{\mathbb{Z}}
\newcommand{\R}{\mathbb{R}}

\newcommand{\from}{\colon}

\newcommand{\ep}{\varepsilon}
\newcommand{\la}{\lambda}

\newcommand{\lapneu}[1]{(-\Delta_{\mathrm{N}})^{#1}}

\newcommand{\into}{\int_{\Omega}}

\newcommand{\de}[1]{\mathrm{d} #1}
\newcommand{\scal}[2]{{\left( #1 , #2 \right)}}
\newcommand{\tM}{{\widetilde{\mathcal{M}}}}

\def\XXint#1#2#3{{\setbox0=\hbox{$#1{#2#3}{\int}$ }
\vcenter{\hbox{$#2#3$ }}\kern-.6\wd0}}

\DeclareMathOperator{\diverg}{div}

\DeclareMathOperator*{\pv}{P.V.}

\newtheorem{proposition}{Proposition}[section]
\newtheorem{theorem}[proposition]{Theorem}
\newtheorem{corollary}[proposition]{Corollary}
\newtheorem{lemma}[proposition]{Lemma}

\theoremstyle{definition}

\newtheorem{remark}[proposition]{Remark}
\newcommand{\beq}{\begin{equation}}
\newcommand{\eeq}{\end{equation}}
\newcommand{\ben}{\begin{enumerate}}
\newcommand{\een}{\end{enumerate}}
\newcommand{\bit}{\begin{itemize}}
\newcommand{\eit}{\end{itemize}}
\newcommand{\dys}{\displaystyle}

\begin{document}

\title[Optimization of the fractional principal eigenvalue]
{Best dispersal strategies in spatially heterogeneous environments: optimization of the principal eigenvalue for indefinite fractional Neumann problems}

\author{Benedetta Pellacci}
\address[B. Pellacci]{Dipartimento di Scienze e Tecnologie,
Universit\`a di Napoli ``Parthenope'', Centro Direzionale, Isola C4 80143 Napoli, Italy.}
\email{benedetta.pellacci@uniparthenope.it}
\author{Gianmaria Verzini}
\address[G. Verzini]{Dipartimento di Matematica, Politecnico di Milano, p.za Leonardo da Vinci 32,  20133 Milano, Italy}
\email{gianmaria.verzini@polimi.it}

\thanks{Research  partially
supported by  MIUR-PRIN-2012-74FYK7 Grant: ``Variational and perturbative aspects of nonlinear differential problems'';
by the ERC Advanced Grant  2013 n. 339958:
``Complex Patterns for Strongly Interacting Dynamical Systems - COMPAT'';
``Gruppo Nazionale per l'Analisi Matematica, la Probabilit\`a e le loro
Applicazioni'' (GNAMPA) of the Istituto Nazionale di Alta Matematica
(INdAM); Parthenope University Project ``Sostegno alla ricerca individuale  2015-2017 e 2016-2018".}

\subjclass[2010]{Primary: 35R11; Secondary: 35P15, 47A75, 92D25.}
\keywords{Spectral fractional Laplacian, reflecting barriers, survival threshold, periodic environments.}

\begin{abstract}
We study the positive principal eigenvalue
of a weighted  problem associated with the Neumann spectral
fractional Laplacian. This analysis is related to the investigation of the
survival threshold in population dynamics. Our  main result concerns
the optimization of such  threshold with respect to the fractional
order $s\in(0,1]$, the case $s=1$ corresponding to the standard Neumann
Laplacian: when the habitat is not too fragmented,
the principal positive eigenvalue can not have local minima for
$0<s<1$.
As a consequence, the best strategy for survival is either following the
diffusion with $s=1$ (i.e. Brownian diffusion), or with the lowest possible $s$ (i.e. diffusion allowing long jumps), depending on
the size of the domain.
In addition, we show that analogous results hold for the
standard fractional Laplacian in $\R^N$, in periodic environments.
\end{abstract}

\date{March 8, 2017}

\maketitle

\section{The Model}\label{sec:model}

Let $u=u(x,t)$ denote the density of a population in position $x$ at time $t$.
The common mathematical model \cite{MR1952568} for the evolution of $u$, in case it
undergoes some kind of dispersal, is given by a reaction-diffusion equation
\[
u_t + L u = f(x,u),
\]
on some spatial domain $\Omega\subset\R^N$, $N\ge1$, with suitable boundary conditions.
The internal reaction $f(x,u)$, which also takes into account the heterogeneity of
the habitat, can take various forms: for our purposes it is sufficient to consider
the simplest case, i.e.  that of a logistic nonlinearity
\[
f(x,u) = m(x) u - u^2,
\]
where the weight $m$ changes sign, distinguishing regions of either favorable or
hostile habitat. The diffusion operator is denoted by $L$, and in principle it
can incorporate a number
of different features of the model. Here, we consider linear,
homogeneous and
isotropic, but possibly non-local, operators.
If the individuals tend to move within the nearest neighborhoods, then
the spatial spread of $u$ is triggered by an underlying random walk of Brownian type,
and it is customary to choose $L=-K\Delta$, for some motility coefficient $K>0$. On
the other hand, in case the resources are sparse, it is expected that more elaborate
hunting strategies, allowing for long jumps, may favor the population survival. Actually,
this guess has also been supported by experimental studies
\cite{viswanathan1996levy,humphries2010environmental}. In this case
the underlying random walk is of Levy flight-type, rather than Brownian, and one is led
to consider fractional diffusion operators, where $-\Delta$ is replaced by $(-\Delta)^s$ \cite{MeKa,MR2584076}.

When  $0<s<1$ the fractional Laplacian in the entire space $\R^N$ can be defined in
different---but equivalent---ways \cite{MR2354493}: for instance via an integral
expression
\[
(-\Delta_{\R^N})^s u (x) = C_{N,s} \pv\int_{\R^N} \frac{u(x)-u(\xi)}{|x-\xi|^{N+2s}}\,\de\xi,
\]
for some dimensional constant $C_{N,s}$,
or as a pseudo-differential operator,  in terms of its Fourier transform:
\begin{equation}\label{eq:fraclap_RN}
\widehat{(-\Delta_{\R^N})^s u} (\xi) = |\xi|^{2s} \widehat{u}(\xi).
\end{equation}
When dealing with a bounded Lipschitz domain $\Omega\subset\R^N$,   the situation is more variegated and different, non-equivalent operators have
been proposed, in dependence of how the boundary
conditions are interpreted.
This complexity of the model is particularly evident
when dealing with Neumann, i.e. no flux boundary conditions,
see \cite{MR3217703, DiRoVa_2014}.
In this paper we consider the boundary as a ``reflecting barrier'',
namely  a barrier
that, in the discrete time counterpart, acts on the long jump
by means of an elastic reflection; this corresponds to the
so-called   ``mirror reflection'' case considered in \cite{MR3217703}.
Reasoning in terms of random walk and imposing the presence of a reflecting barrier on
$\partial\Omega$, one is led, at least heuristically, to consider the Neumann spectral fractional Laplacian, i.e.
\begin{equation}\label{def:lapneu}
\lapneu{s} u = \sum_{k=1}^{\infty} \mu_k^s \left(\int_\Omega u \phi_k\,dx\right)\phi_k,
\end{equation}
where $0=\mu_0<\mu_1\leq\mu_2\leq\dots$ denote the Neumann eigenvalues of $-\Delta$ in $H^1(\Omega)$, and
$(\phi_k)_k$ the corresponding normalized eigenfunctions. This operator has been considered
in different models and applications, see for instance \cite{cast,MR3385190,DiSoVa_2015}.
The relation between the Neumann spectral fractional Laplacian
and random walks with long jumps and reflections has been discussed in \cite{MoPeVe}, in dimension $N=1$, and
those arguments can be easily extended to higher dimensions in case $\Omega$ is a rectangle:
in fact, the correspondence holds true as far as the reflecting barrier can be treated by
the method of images, by introducing reflected domains in which the motion can be continued,
and then by quotienting by the symmetries.
Several other interpretations of the boundary as a reflecting barrier
are available in the literature: for instance, in \cite{MR3467345}
the barrier acts on the long jump by just stopping the particle at the
boundary, without any rebound;  in
\cite{DiRoVa_2014},  also the action of the boundary
is not deterministic.

Incidentally,  another established point of view is that of
dealing with a  periodically fragmented environment in $\R^{N}$ \cite{MR2214420,beroro,MR3477397,CaDiVa_2016}. Actually,
for our purposes, the treatment of the periodic model is very similar to
that with  mirror reflections.
Indeed, we can also deal with the fractional Laplacian on the
whole $\R^N$, instead
of the Neumann spectral one, by assuming that the environment is
periodic.

Another controversial feature of the model we are describing regards the form of the generalized diffusion coefficient:
a number of contributions deals with the difficulty of properly defining
(and measuring) the motility coefficient $K$  \cite{habeav,wube,veetal}.
Motivated  by dimensional arguments and modeling ones,
in this context $K$ is supposed to depend on $s$, and
a commonly accepted expression for it has been introduced in
\cite[Section 3.5]{MeKa} as
\[
K(s) = \frac{\sigma^{2s}}{\tau}
\]
(see also \cite{Ha,Ha2}), where the scales $\sigma$ and $\tau$ are respectively characteristic length and time associated
with the diffusion process. Without loss of generality, we
can choose via time scaling  $\tau=1$ and write $d=\sigma^2>0$. Summing up, we consider the equation
\begin{equation}\label{eq:diff_eq_neu}
\partial_t u + d^s \lapneu{s} u = m(x)u - u^2, \qquad x\in\Omega,
\end{equation}
for $0<s\leq 1$, where $s=1$ corresponds to the case of standard
(local) diffusion.
Alternatively, one may deal with a modified version of
\eqref{eq:diff_eq_neu}, in which the motility coefficient does not
depend on $s$:  we can also manage this case, as it actually can be
seen as a particular case of \eqref{eq:diff_eq_neu} when $d=1$, see
Section \ref{sec:mainres}.

A main question related to \eqref{eq:diff_eq_neu} concerns survival of the population, that is, the identification
of conditions (on $\Omega,s,d,m$) which imply that solutions to  \eqref{eq:diff_eq_neu} do not vanish asymptotically for
$t\to+\infty$. When $s=1$, it is well known that such conditions are related to the existence of a positive steady state, which attracts
every non-negative non-trivial solution. In turn, the existence of such steady state can be expressed in terms of the principal eigenvalue of the associated linearized problem \cite{MR1100011,MR1105497,MR2214420}.
These results can also be extended to the fractional setting \cite{beroro,kalosh,MoPeVe}.

Taking $m\in L^\infty(\Omega)$, two different
situations may occur in dependence on its average: if $m$
has non-negative average (and it is non-trivial) then there is always survival. On the other hand, in the case
\begin{equation}\label{ipom}
m \in \mathcal{M}:= \left\{ m \in L^\infty(\Omega):\int_\Omega m <0,\, m^+\not\equiv0\right\},
\end{equation}
the survival is related to the weighted eigenvalue problem
\begin{equation}\label{eq:eigen_u}
d^s \lapneu{s} u = \lambda m u, \qquad x\in\Omega.
\end{equation}
More precisely,  in Appendix \ref{app:eig} we show  that,
under condition \eqref{ipom},
there exists a unique positive principal eigenvalue
\[
\la_1=\la_1(m,d,s)>0,
\]
with a positive eigenfunction. Moreover, reasoning as in \cite[Theorem 1.2]{beroro}, one has
that solutions to \eqref{eq:diff_eq_neu} survive (i.e. they tend to the unique positive steady state, as $t\to+\infty$) if and only if $\lambda_1(m,d,s) < 1$.
Then, natural questions concern the dependence
of $\lambda_{1}$ on the parameters of the problem, and in particular its optimization.
Note that, through a  change of variables, rescaling the size of the domain is
equivalent to rescaling the diffusion coefficient $d$ while keeping $\Omega$ fixed. Here we
choose this second point of view, and this is the reason why we do not consider explicitly
the dependence of the eigenvalue on the domain.

In the case $s=1$ of standard diffusion, the dependence of $\lambda_1$ on $d$ can be easily
scaled out and the eigenvalue actually depends only on $m$. Accordingly, the problem of
minimizing $\lambda_1$
has been mainly considered, when $m$ varies
within a suitable admissible class,  see
\cite{MR1014659,MR1112065,MR1105497,MR1858877,loya,MR2364810,MR2494032,dego} and references
therein. The typical result obtained is that the minimizer $m$ exists and it is of
bang-bang type (i.e. it coincides with its maximum value $\overline m>0$ on some $D\subset
\Omega$, and with its minimum $-\underline m<0$
on $\Omega\setminus D$). Furthermore, the best environment has a few
number of relatively large favorable regions. As observed in \cite{MR1014659}, this has
significant implications for the design of wildlife refugees. Part of these results can also be
extended to the case $s<1$, as discussed in Section \ref{sec:opt_m} below, but our main interest in
the present paper is to analyze the properties of the map
\[
(m,d,s) \mapsto \lambda_1(m,d,s),
\]
aiming at optimizing $\lambda_1$, mainly with respect to $0<s\leq 1$. From a biological
viewpoint, this amounts to wonder whether, for given population and habitat, the Brownian hunting
strategy is more effective than the long jumps one, in order to survive. The good starting point
in our analysis is that the map $s\mapsto \lambda_1(m,d,s)$ is smooth in $(0,1]$ (see Appendix \ref{app:eig}).
Up to our knowledge,  there are very few contributions
concerning  the optimization  of the  order $s$
in fractional diffusion equations; in particular, a related but
different problem has been considered in \cite{SpVa_2016}.

It is worth noticing that  part of the cited above literature does not treat exactly problem
\eqref{eq:eigen_u} (with $s=1$), but rather the related version
\begin{equation}\label{eq:beresteigen}
-d\Delta u - m u = \tilde\lambda u, \qquad x\in\Omega.
\end{equation}
It is easy to show that $\lambda_1<1 $ if and only if
$\tilde\lambda_1<0$, therefore both these eigenvalues play
analogous roles for  survival. One main advantage of the latter problem is that a principal
eigenvalue $\tilde\lambda_1$ always exists, regardless of the average of $m$; on the other
hand, we prefer to deal with \eqref{eq:eigen_u}
because, among other properties, the dependence of $\lambda_1$
on $d$ can be treated in a simpler way.

As we mentioned, the case of the fractional laplacian on the full space, with a periodic environment, 
is of interest too.
More precisely,
following \cite{beroro}, let us introduce the hyperrectangle
$\mathcal{C}_{l}=(0,l_{1})\times
\cdots \times(0,l_{N})\subset\R^N$, and let us assume that
\[
m:\R^N\to\R\text{ is $\mathcal{C}_{l}$-periodic.}
\]
In case $m|_{\mathcal{C}_l}$ satisfies \eqref{ipom} (with $\Omega$
replaced by $\mathcal{C}_l$), we have the existence of a positive
principal eigenvalue $\lambda_{\mathrm{per}}=
\lambda_{\mathrm{per}}(m,d,s)$, with positive periodic
eigenfunction, for the problem
\begin{equation}\label{eq:period_eigen}
d^s (-\Delta_{\R^N})^s u = \lambda m u, \qquad x\in\mathbb{T}^N:=\R^N/\mathcal{C}_l.
\end{equation}
Moreover, the solutions to the problem
\[
\partial_t u + d^s (-\Delta_{\R^N})^s u = m(x)u - u^2, \qquad x\in\R^N,
\]
where no periodicity condition is assumed on $u$, survive if and only if
$\lambda_{\mathrm{per}}<1$. Actually, these results are proved in \cite{beroro} in terms
of the eigenvalue $\tilde\lambda_{\mathrm{per}}$ corresponding to \eqref{eq:beresteigen},
but the two conditions can be easily proved to be equivalent.
Now, it is easy to be convinced that
in some particular cases the Neumann spectral eigenvalue problem
\eqref{eq:eigen_u} and
the periodic one \eqref{eq:period_eigen} are equivalent. For instance, if
$m$ is defined
in a hyperrectangle $\Omega$, then one can extend it to $2\Omega$ by
even reflection, and
then to $\R^N$ by periodicity; hence, using the uniqueness properties of
the principal eigenfunctions, one can reduce the Neumann problem in
$\Omega$ to the periodic one; the opposite reduction can be done too, in
case $m$ is $\mathcal{C}_l$-periodic,
and even with respect to the directions of its sides (up to translations).
However, also for general $\Omega$ and $m$, the two problems
share the same structure.
Indeed, let
$\mathcal{C}_l$ be fixed and
let $(\nu_{k})_k$, $(\varphi_{k})_k$ denote the periodic eigenvalues
and eigenfunctions of $-\Delta$ in $\mathcal{C}_l$
(which can be explicitly computed).
Then we can introduce
the periodic spectral fractional Laplacian as
\begin{equation}\label{eq;periodic_fraclap}
(-\Delta_{\mathrm{per}})^{s} u = \sum_{k=1}^{\infty} \nu_k^s \left(\int_{\mathcal{C}_l}
u \varphi_k\,dx\right)\varphi_k.
\end{equation}
The following result (which is a version of Theorem A in \cite{ancora_Stinga}) allows to connect spectral operators with periodic ones.
\begin{proposition}\label{prop:periodic_is_spectral}
If $u$ is continuous and $\mathcal{C}_l$-periodic, then
\[
(-\Delta_{\R^{N}})^{s} u= (-\Delta_{\mathrm{per}})^{s}u, \qquad x\in\mathcal{C}_l.
\]
\end{proposition}
As a consequence, our techniques apply to this setting too.
\smallskip

The paper is structured as follows. In the next section we present our
main results about the optimization of $\lambda_1$ with respect to $s$,
joint with their biological interpretations.
The proofs of such results are given in Section \ref{preliminari}.
In Section \ref{sec:opt_m} we briefly discuss the optimization of $\lambda_1$ with respect to $m$ and in Section \ref{sec:final}
we provide some discussion, including a recap of the biological and mathematical significance of our results, as well as some numerical simulations to better illustrate our
assumptions.
We postpone the proofs of the existence and regularity properties of eigenvalues and eigenfunctions associated with \eqref{eq:eigen_u}
in Appendix \ref{app:eig} and we collect the proofs of  some
technical results exploited in the paper in Appendix \ref{app:aux}.

\subsection*{Acknowledgments}
We would like to thank Alessandro Zilio for his precious help with the numerical simulations, and Marco Fuhrman and Sandro Salsa
for the fruitful discussions.

\subsection*{Notation.} We write $\scal{\cdot}{\cdot}$ for the
scalar product in $L^2(\Omega)$.
We will denote with $\phi_{k}$ the eigenfunctions of the classical
Laplace operator  in $\Omega$ with Neumann homogeneous boundary
conditions, normalized in $L^2(\Omega)$. Their associated eigenvalues will be denoted by
$\mu_{k}$. For a function $u\in L^2(\Omega)$, we write
\begin{equation}\label{not:muk}
u = u_0 + \sum_{k=1}^{\infty} u_k \phi_k,
\qquad\text{
where }
u_0 = \frac{1}{|\Omega|}\int_\Omega u,\quad
u_k = \scal{u}{\phi_k},\,k\geq 1,
\end{equation}
where $|\Omega|$ stands for the Lebesgue measure of the set
$\Omega$.
Often we will write $u=u_0 + \tilde u$.
Finally, $C$ denotes every (positive) constant we do not need  to specify, whose value may change, 
also within the same formula.

\section{Main Results }\label{sec:mainres}
Throughout all the paper we will assume, up to further restrictions, $\Omega\subset\R^N$ to be a bounded, Lipschitz domain, $d>0$, $0<s\leq 1$, and $m\in\mathcal{M}$, defined in \eqref{ipom}.
The starting observation in our optimization results is the scaling property
\[
\lambda_1(m,d,s) = d^s \lambda_1(m,1,s),
\]
which allows us to prove that, if $d$ is very large or very small, with respect to
the size of $\Omega$, then the map $s\mapsto
\lambda_1(m,d,s)$ becomes  monotone, and therefore it is minimized either for $s=1$ or for
$s$ small. More precisely, recalling that $\mu_1>\mu_{0}=0$ denotes the first positive Neumann eigenvalue
of $-\Delta$ in $H^1(\Omega)$, we have the following.
\begin{theorem}\label{prop:dep_on_d}
Let $\Omega$ and $m$ be fixed and satisfying \eqref{ipom}. Then:
\begin{itemize}
 \item if $d\geq \dfrac{1}{\mu_1}$ then the map $s \mapsto \lambda_1(m,d,s)$ is monotone
 increasing in $(0,1]$;
 \item for any $0<a<1$ there exists $\underline{d}>0$, depending only on $a$, $\Omega$
 and $m$,
 such that if $d\leq\underline{d}$ then the map $s \mapsto \lambda_1(m,d,s)$ is monotone decreasing in $[a,1]$.
\end{itemize}
\end{theorem}
\textbf{Biological Interpretation.}
The motility coefficient $d$ is related to a characteristic length associated
with the diffusion process,  and in particular a small $d$ corresponds to
the case of a domain which is large, with respect to the
diffusion characteristic length, and vice versa. As
a consequence,  Theorem \ref{prop:dep_on_d} states that
in very large environments the local diffusion is more successful, while in very small ones a fractional diffusion strategy would be preferable. Similar effects in related models were already
noticed in \cite{CaDiVa_2016}, Theorem 1.5. As observed in Section \ref{sec:model}, one
may consider as a motility coefficient $d$ instead of $d^{s}$. In such a case, we have that the corresponding eigenvalue depends linearly on $d$, therefore its monotonicity properties
(w.r.t. $s$) are not affected by changing the motility. 

\medskip

From Theorem \ref{prop:dep_on_d} it is clear that, when $d$ increases from $\underline{d}$
to $1/\mu_1$, then the map $s \mapsto \lambda_1(m,d,s)$ has a transition from decreasing to
increasing, and therefore it develops internal critical points.
The main result we obtain
in this paper is  the following.
\begin{theorem}\label{thm:main_intro}
Let $M,\rho,\delta$ be positive constants and set
\begin{equation}\label{eq:mconlapalla}
\tM:= \left\{ m \in \mathcal{M} :
m\geq -M,\,\exists\, B_\rho(x_0)\subset \Omega\text{ with } m|_{B_\rho(x_0)}\geq\delta
\right\}.
\end{equation}
Let us assume that  $m\in\tM$ and that one of the following conditions holds:
\begin{itemize}
 \item either
\begin{equation}\label{eq:mainassnew}
\frac{\delta}{M} \cdot \rho^2\mu_1 > \lambda_1^{\text{Dir}}(B_1),
\end{equation}
 where $\lambda_1^{\text{Dir}}(B_1)$ is
 the first eigenvalue of $-\Delta$ with homogenous Dirichlet boundary  conditions  of the  ball of radius $1$;
 \item or  $M,\rho,\delta$ are arbitrary, and
 \begin{equation}\label{eq:mainassA}
-A< \int_\Omega m <0,
\end{equation}
where $A=A(M,\rho,\delta,N)>0$ is an explicit constant.
\end{itemize}
Then, depending on $d>0$, either the map $s \mapsto \lambda_1(m,d,s)$ is monotone,
or it has exactly one maximum in $(0,1)$. In particular,
the limit $\lambda_1(m,d,0^+)$
is well defined for every $d$, and
\[
\inf_{0<s\leq 1} \lambda_1(m,d,s) =
\begin{cases}
\lambda_1(m,d,1) & \text{when }0<d\le d^*\\
\lambda_1(m,d,0^+) & \text{when }d\ge d^*,
\end{cases}
\]
where $d^*=\lambda_1(m,1,0^+)/\lambda_1(m,1,1)$.
\end{theorem}
\textbf{Biological Interpretation.}
Let us first comment hypotheses
\eqref{eq:mconlapalla}, \eqref{eq:mainassnew}, \eqref{eq:mainassA}.
Non-degeneracy conditions as the one present in  \eqref{eq:mconlapalla}
have been already considered  in the literature to avoid an excessive
fragmentation of the favorable region (see for instance
Theorem 3.1 in \cite{MR1014659}). Indeed such an assumption insures
the existence of an enclave having a fixed minimal size, measured by $
\rho$, in which the growth rate is at least
$\delta$. From this point of view, assumption \eqref{eq:mainassnew}
requires that $\rho$ and  $\delta$ are not too small, with respect to the
other parameters. Notice that both the quantities
$\delta/M$ and $\rho^2\mu_1$ are scale invariant: the first is the ratio
between the least guaranteed  favorable growth rate and the worst
unfavorable one; the second compares the sizes of the enclave
and of $\Omega$, respectively (recall that, being $\mu_1$ an
eigenvalue  of $\Omega$,
$\mu_1(t\Omega) = t^{-2}\mu_1(\Omega)$).
Moreover, being $\lambda_1^{\text{Dir}}(B_1)  $ in
\eqref{eq:mainassnew} the first eigenvalue of $-\Delta$ with homogenous Dirichlet boundary  conditions  in the  ball with unitary radius,  it is actually an explicit universal constant only
depending on the dimension $N$ and it can be explicitly
computed. In particular, in dimension 2 or 3
\[
\lambda_1^{\text{Dir}}(B_1)  \simeq 5.78, \qquad \lambda_1^{\text{Dir}}(B_1)  \simeq 9.87.
\]
On the
other hand, assumption \eqref{eq:mainassA} admits
the environment to be possibly severely disrupted, once its average is
sufficiently close to $0$, being above an explicit lower bound
given by
\begin{equation}\label{eq:Aesp}
A = \left(\frac{2|\partial B_1|}{\lambda_1^{\text{Dir}}(B_1)}\right)^{1/2}
\frac{\mu_1 \delta^2\rho^{2+N/2}}{M\lambda_1^{\text{Dir}}(B_1)-\mu_1 \delta\rho^2}.
\end{equation}
Under this perspective, Theorem \ref{thm:main_intro} says that, if
\begin{itemize}
 \item either the habitat is not too  fragmented,
 \item or it is fragmented, but not too hostile in average,
\end{itemize}
then the best choice of $s$ is always either the smallest admissible value, or the biggest one.
This extends the consequences of Theorem \ref{prop:dep_on_d} also to the case
when the sizes of $d$ and $\Omega$ are comparable (beyond that of both $d$ and $s$ small). 
As observed in Section \ref{sec:model}, one
may consider as a motility coefficient $d$ instead of $d^{s}$.
 Then, dividing by $d$ and relabelling $\lambda$, one can easily obtain a version
of Theorem \ref{thm:main_intro} also for this problem. In particular, also in this case $d$ does
not affect the monotonicity properties of $\lambda$ and for any $m$ satisfying the assumptions
of the theorem, the optimal strategy is either $s=1$ or $s=0^+$.
\medskip

As we explained in Section \ref{sec:model},
problems \eqref{eq:eigen_u} and \eqref{eq:period_eigen}
enjoy the same structure. Indeed, thanks to Proposition
\ref{prop:periodic_is_spectral}, we can interpret the fractional laplacian on the full space, in a periodic environment, as a spectral operator. Hence
we can extend the analysis of the Neumann problem also to the periodic case.
In this respect, the key observation is that the spectrum of the
Neumann problem and that of the periodic one share the same main
properties, namely, they both consist in  a diverging sequence of eigenvalues, with first,
simple element $\mu_0=\nu_0=0$, and they both
are associated with a basis of eigenfunctions which are orthogonal in $H^1$ and
orthonormal in $L^2$.
Then all the results for the Neumann case also hold true in the periodic
one. In particular, we have the following counterpart of Theorem
\ref{thm:main_intro} in the periodic setting.
\begin{theorem}\label{thm:main_intro_per}
Let $m$ be $\mathcal{C}_l$-periodic. If $m|_{\mathcal{C}_l}$ satisfies the assumptions of
Theorem \ref{thm:main_intro} then, for every $d>0$, the map $s \mapsto \lambda_{\mathrm{per}}(m,d,s)$ is either monotone
or it has exactly one internal maximum in $(0,1)$, and
\[
\inf_{0<s\leq 1} \lambda_{\mathrm{per}}(m,d,s) =
\begin{cases}
\lambda_{\mathrm{per}}(m,d,1) & \text{when }0<d\le d^*\\
\lambda_{\mathrm{per}}(m,d,0^+) & \text{when }d\ge d^*,
\end{cases}
\]
where $d^*=\lambda_{\mathrm{per}}(m,1,0^+)/\lambda_{\mathrm{per}}(m,1,1)$.
\end{theorem}
\textbf{Biological Interpretation.}
In a periodically fragmented habitat, when the dispersal of the
population is triggered by a random  walk of Levy flight-type, the
same conclusions of Theorem \ref{thm:main_intro} hold true:
depending on the features of the habitat, namely not
excessive fragmentation or not too  hostility in average,
the best strategy of diffusion is  still either the local one
(i.e. $s=1$) or the smallest fractional.
\medskip

Let us point out that with our techniques we can
also deal  with other
fractional spectral operators. For instance, the Dirichlet case can be treated
in an even easier way, since in such case zero is not an eigenvalue of $-\Delta$. More
generally, we can also deal with the fractional Laplace-Beltrami operator on a compact
Riemannian manifold, with or without boundary (with appropriate boundary conditions in the former
case). From this point of view, of course, the periodic case described above corresponds to the
flat torus.

A natural question is to wonder whether or not
the assumptions in Theorems \ref{thm:main_intro}, \ref{thm:main_intro_per}
on the environment are merely  technical
 and the result may hold for more general $m$.
In Section \ref{sec:final} we provide some simple
numerical simulations which suggest that this is not the case and that the map
$s \mapsto \lambda_1(m,d,s)$ may present interior minima, as well as multiple
local extrema (see Figure \ref{fig:result} ahead).

\section{Proofs of The main results}\label{preliminari}

In this section we will provide the proofs of the main results of the paper:
Theorem \ref{prop:dep_on_d}, Theorem \ref{thm:main_intro} and
Theorem \ref{thm:main_intro_per}.  Our arguments will be based on
the study of the positive principal eigenvalue $\la_{1}(m,d,s)$,
which  can be characterized as follows
\begin{equation}\label{eq:def_eig}
\lambda_{1}(m,d,s)
=d^{s}\min_{ H^s(\Omega)}
\left\{\scal{ \lapneu{s}u}{u}:
\into mu^{2}=1
\right\}.
\end{equation}
The proof of the existence and the main properties of
$\lambda_{1}(m,d,s)$ are collected in Appendix \ref{app:eig}.

The following estimates will be exploited in our analysis.

\begin{proposition}\label{prop:lambda1dasopra}
The eigenvalue $\lambda_{1}(m,d,s)$ satisfies the following estimates
\begin{align}\label{alto}
 \la_{1}(m,d,s)&\leq (d\mu_1)^{s-1} \lambda_1(m,d,1),
\\
\label{basso}
 \la_{1}(m,d,s)&\geq  \frac{\dys d^{s}\mu_{1}^{s}\left|\into m\right|}{\dys\sup_{\Omega}m\left|\into m\right|+\|m\|_{L^2}^{2}}.
\end{align}
\end{proposition}
\begin{proof}
Let $\psi_1$ denote the first normalized eigenfunction associated with $\lambda_1(m,d,1)$. Then
\[
\begin{split}
\lambda_1(m,d,s) &\leq \scal{d^s\lapneu{s}\psi_1}{\psi_1} = \sum_k (d\mu_k)^s(\psi_1)_k^2 =
\sum_k (d\mu_k)^{s-1}d\mu_k(\psi_1)_k^2 \\
& \leq (d\mu_1)^{s-1}\sum_k d\mu_k(\psi_1)_k^2 =
(d\mu_1)^{s-1} \lambda_1(m,d,1).
\end{split}
\]
In order to show \eqref{basso},
we follow ideas introduced, for $s=1$, in \cite{Saut}. To start with,
notice that, using \eqref{not:muk},
the following Poincar\'e inequality holds:
\begin{equation}\label{eq:poincar}
u\in H^s(\Omega),\;u_0= 0
\quad\implies\quad
\int_\Omega u^2 \leq \frac{1}{\mu_1^s} \scal{\lapneu{s} u}{u}.
\end{equation}
Indeed,
\[
\int_\Omega u^2 = \sum_{k\geq 1} u_k^2 \leq \frac{1}{\mu_1^s}\sum_{k\geq 1}\mu_k^s u_k^2
=\frac{1}{\mu_1^s} \scal{\lapneu{s} u}{u}.
\]
Using the decomposition $\psi_1= h +\tilde{ \psi_1} $, with $h\in\R$ and $\tilde{ \psi_1 }$
with zero average, we can exploit the fact that $\lapneu{s} \psi_1$ has zero average to infer
\[
 0=\la_{1}(m,d,s)\into m(h +\tilde{ \psi_1 })
 \quad\implies\quad
 h=-\frac{\into m\tilde{\psi}_{1}}{\into m}.
\]
Then
\begin{align*}
d^{s}\scal{\lapneu{s} \psi_1}{\psi_1}
&=\la_{1}(m,d,s)\into m(h +\tilde{ \psi_1 })^{2}
\\
&=
\la_{1}(m,d,s)\left\{
\into m\tilde\psi^{2}_{1}-\frac1{\into m}
\left[\into m\tilde{\psi}_{1}\right]^{2}\right\}
\\
&\leq
\la_{1}(m,d,s)\left\{
\sup_{\Omega}m +\frac1{\left|\into m\right|}\|m\|_{L^2}^{2}
\right\}\int_\Omega \tilde\psi_1^2
\\
&\leq
\frac{\la_{1}(m,d,s)}{{\mu_{1}^{s}}}\left\{
\sup_{\Omega}m +\frac1{\left|\into m\right|}\|m\|_{L^2}^{2}
\right\}\scal{\lapneu{s}\psi_1}{\psi_1},
\end{align*}
where we used \eqref{eq:poincar} in the last line. Then \eqref{basso} holds.
\end{proof}
{\bf Biological Interpretation.}
Proposition \ref{prop:lambda1dasopra} furnishes bounds on the survival
threshold in terms on the survival threshold for $s=1$, the
motility coefficient and the domain. In particular, it
implies that the survival threshold can neither vanish nor blow up, even when
$s$ approaches zero.
\medskip

In the following,
 $m\in\mathcal{M}$ is fixed, therefore, for easier notation,
we omit the dependence on $m$ and write $\lambda_1=\lambda_1(d,s)$.
We first observe that, according to the characterization \eqref{eq:def_eig},
the parameter $d$ only affects the first eigenvalue
$\lambda_1$ and not the corresponding eigenfunction.
\begin{lemma}\label{lem:d1d2}
For any $d_1,d_2\in\R^+$ we have that
\[
\lambda_1(d_2,s)=\frac{d_2^s}{d_1^s}\lambda_1(d_1,s),
\]
and both eigenvalues share the same eigenfunction.
Furthermore, for any fixed $s\in(0,1]$, the map
\[
d \mapsto \lambda_1(d,s)
\]
is monotone increasing.
\end{lemma}
\begin{proof}
The lemma is an immediate consequence of Propositions \ref{prop:lambdak} and
\ref{prop:lambda1}. Indeed, let $\psi_1$ be the first positive normalized
eigenfunction associated with $\lambda_1(d_1,s)$. Then
\[
d_2^s \lapneu{s} \psi_1 = \frac{d_2^s}{d_1^s}
\left(d_1^s \lapneu{s} \psi_{1} \right)=
\frac{d_2^s}{d_1^s}\lambda_1(d_1,s)m \psi_1.
\]
Therefore, $\psi_{1}$ is also a positive eigenfunction associated
with the eigenvalue $\nu=d_{2}^{s}\la_1(d_{1},s)/d_{1}^{s}$. Taking into account that
$\lambda_1(d_2,s)$ is simple we obtain that $\la_{1}(d_{2},s)=d_{2}^{s}\la_1(d_{1},s)/d_{1}^{s}$,
with same eigenfunction as $\la_1(d_{1},s)$.

In particular, we obtain that, for any $d>0$,
\begin{equation}\label{eq:lambda11}
\lambda_1(d,s) = d^s \lambda_1(1,s),
\end{equation}
so that the last statement  easily  follows, as $\la_{1}(1,s)$ is positive.
\end{proof}
\begin{proof}[Proof of Theorem \ref{prop:dep_on_d}]

Let $d\mu_1\geq 1$ and $s_1<s_2$. For any fixed $u$, we have
\[
\lambda_1(d,s_1)\le d^{s_1}\scal{\lapneu{s_1} u}{u} = \sum_k (d\mu_k)^{s_1}u_k^2 \leq \sum_k (d\mu_k)^{s_2}u_k^2 =
d^{s_2}\scal{\lapneu{s_2} u}{u}.
\]
Recalling \eqref{eq:def_eig}, we obtain that $\lambda_1$ is non-decreasing in $s$. To conclude,
let us assume by contradiction that $\lambda_1(d,s_1)=\lambda_1(d,s_2)$, and let $u$ denote
the first eigenfunction associated with $\lambda_1(d,s_2)$. As $\int_\Omega mu^2=1$, by the above inequality we deduce
that $u$ achieves also $\lambda_1(d,s_1)$; thus $d\mu_1=1$ and
$u_k=0$ whenever $k\geq\nu_1+1$, where $\nu_1$ is the multiplicity of $\mu_1$ as a Neumann eigenvalue of $-\Delta$. As a consequence,
\[
\lambda_1(d,s_1)mu = d^{s_1}\lapneu{s_1} u = \sum_{k=1}^{\nu_1} (d\mu_1)^{s_1}u_k\phi_k  =\sum_{k=1}^{\nu_1} u_k\phi_k =u,
\]
yielding a contradiction since $m$ is not constant.

In order to prove the second conclusion of Theorem \ref{prop:dep_on_d}, let us
differentiate \eqref{eq:lambda11} with respect to $s$ to obtain
\[
\frac{\partial}{\partial s}\lambda_1(d,s) = d^s \left[(\log d) \lambda_1(1,s)
+ \frac{\partial}{\partial s}\lambda_1(1,s)\right].
\]
By Propositions \ref{prop:lambda1} and \ref{prop:lambdak}, for any $a\in (0,1)$ the map $s\mapsto \lambda_1(1,s)$ is $C^1([a,1])$, and $\lambda_1(1,s)>0$ for $s\in [a,1]$.
Then, choosing
\[
\underline{d} = \min_{s\in[a,1]}\exp\frac{-\frac{\partial}{\partial s}\lambda_1(1,s)}{
\lambda_1(1,s)},
\]
the conclusion easily follows.
\end{proof}
\begin{remark}
Analogously, one can show that, for any $0<a<1$,
\[
\frac{\partial}{\partial s}\lambda_1(d,s)>0 \text{ for every }s\in[a,1]
\]
whenever
\[
d>\overline{d} = \max_{s\in[a,1]}\exp\frac{-\frac{\partial}{\partial s}\lambda_1(1,s)}{
\lambda_1(1,s)}.
\]
Note that the first conclusion of Theorem \ref{prop:dep_on_d} implies that $\overline{d} \leq 1/\mu_1$, for every
$a\in(0,1)$, implying
\[
\frac{\partial}{\partial s}\lambda_1(1,s) \geq \ln(\mu_1) \lambda_1(1,s), \qquad s\in(0,1].
\]
This provides a uniform condition on $d$ in order to have $\lambda_1(d,s)$
increasing for $s\in(0,1)$. On the other hand, we are not able to give an analogous
uniform assumption implying the opposite monotonicity.
\end{remark}

Next we turn to the study of the intermediate values of $d$, when the map
$s\mapsto \lambda_{1}(d,s)$ has a transition in its monotonicity properties.
In the following, we will denote the normalized first eigenfunction associated with $\lambda_1(d,s)$,
which does not depend on $d$, as $\psi_s$:
\begin{equation}\label{eq:psi_1}
d^s \lapneu{s} \psi_{s} = \lambda_1(d,s) m \psi_{s},\qquad \int_\Omega m \psi_s^2=1.
\end{equation}
The following two lemmas introduce the mathematical tools
exploited in the proof of Theorem \ref{thm:main_intro}.
We include here the statements for the reader's convenience
and we postpone the proofs in Appendix \ref{app:aux}.
\begin{lemma}\label{lem:curvetta}
Let $\bar s \in (0,1)$, $0<\ep<\bar s$, and $w\in H^{\bar s + \ep}(\Omega)$
be such that
\begin{equation}\label{eq:assumpt_curvetta}
\int_\Omega m \psi_{\bar s} w=0.
\end{equation}
Then there exists a $C^2$ curve $u:(\bar s-\ep,\bar s+ \ep)  \mapsto H^{\bar s + \ep}(\Omega)$ such that
\beq\label{equality:v}
u(\bar s) = \psi_{\bar s},\qquad \dot u(\bar s) = w,\qquad \int_\Omega m u^2(t) = 1\;\text{ for every }t.
\eeq
Furthermore
\beq\label{equation:ddotu}
 \scal{ d^{\bar s}\lapneu{\bar s} \psi_{\bar s}}{ \ddot u(\bar s)}
= -\lambda_1(d,\bar s)\int_\Omega m w^2.
\eeq
\end{lemma}
Taking into account the spectral decomposition of  any $w\in H^{s+
\ep}(\Omega)$,
\[
w=\sum_{k=0}^{\infty}w_{k}\phi_{k},
\]
and recalling \eqref{def:lapneu}, we can define the following
operators, derivatives of $d^{s}\lapneu{s}$
with respect to $s$:
\begin{equation}\label{eq:LeT}
\begin{split}
  L_{s}(w)&:=\partial_{s}\left[d^{s}\lapneu{s}\right](w)=\sum_{k=1}^{\infty}(d\mu_{k})^{s}\ln(d\mu_{k})w_{k}\phi_k,\\
  T_{s}(w)&:=\partial^2_{ss}\left[d^{s}\lapneu{s}\right](w)=\sum_{k=1}^{\infty}(d\mu_{k})^{s}\ln^2(d\mu_{k})w_{k}\phi_k.
\end{split}
\end{equation}
\begin{lemma}\label{lem:vtilde}
For $\psi_s$ as in \eqref{eq:psi_1} there exists a unique $v\in H^s(\Omega)$, solution of the problem
\beq\label{problem:tildev}
d^s\lapneu{s} v= L_{s}\psi_{s},\qquad \int_\Omega m v =0.
\eeq
Furthermore $v \in H^{s'}(\Omega)$ for every $s'<2s$, and
\begin{equation}\label{eq:vertice}
\int_\Omega mv^2 = \max_{c\in\R} \int_\Omega m(c+v)^2.
\end{equation}
\end{lemma}
The function $v$ above will be crucial in
proving the following result, which introduces the abstract condition to
obtain our main results.
This condition will involve, in addition to $\la_{1}(m,d,s)$, the
eigenvalue $\la_{-1}(m,d,s)$
(whose existence is proved in Appendix \ref{app:eig})
characterized as follows
\beq\label{def:la-1}
-\la_{-1}(m,d,s)
=d^{s}\min_{H^s(\Omega)}\left\{\scal{ \lapneu{s}u}{u}
:
\into mu^{2}=-1,\,\into mu=0
\right\}.
\eeq
Notice that $\la_{-1}$ is actually a second eigenvalue and it may
not be simple. Indeed,
since $m$ has negative mean, the constant eigenfunction $\psi_0$ associated
with $\lambda_0=0$ satisfies
\[
\int_\Omega m \psi_0^2 <0.
\]
As a consequence, without imposing that $u$ is orthogonal to $m$ with respect to the $L^{2}(\Omega)$ scalar product,
the minimization problem \eqref{def:la-1}
has the solution
$\la_{0}=0$.
\\
We are now ready to state our abstract result.
\begin{theorem}\label{thm:squarepos}
Let $d>0$ be fixed and $\bar s \in (0,1)$ be such that
\begin{equation}\label{eq:critpt}
\frac{\partial}{\partial s} \lambda_1(d,\bar s)=0.
\end{equation}
If
\begin{equation}\label{eq:la1la-1}
-\la_{-1}(1,\bar s)>\la_{1}(1,\bar s),
\end{equation}
then $\bar s$ is a point of local maximum of the map $s\mapsto \lambda_1(d,s)$.
\end{theorem}
\begin{proof}
Let us first note that, using \eqref{eq:psi_1} and \eqref{eq:critpt}, we have
\beq\label{critical:s}
0=\frac{\partial}{\partial s}\scal{d^s \lapneu{s} \psi_{s}}{\psi_{s}}|_{s=\bar s} =
2\scal{ d^{\bar s}\lapneu{\bar s}\psi_{\bar s}}{\dot \psi_{\bar s}}
+ \scal{ L_{\bar s}\psi_{\bar s}}{\psi_{\bar s}},
\eeq
where $\dot\psi_s = (d/ds) \psi_s$ and $L_{s}$ is defined in \eqref{eq:LeT}. We infer
that
\[
\begin{split}
\scal{ L_{\bar s}\psi_{\bar s}}{\psi_{\bar s}}& =-2\scal{  d^{\bar s}\lapneu{\bar s}\psi_{\bar s}}{\dot \psi_{\bar s}}=
-2\lambda_1(d,\bar s) \int_\Omega m \psi_{\bar s}\dot \psi_{\bar s}\\ & = -\lambda_1(d,\bar s) \frac{d}{ds} \int_\Omega m \psi_{s}^2
{\Big |}_{s=\bar{s}} =0.
\end{split}
\]
For $v$ as in Lemma \ref{lem:vtilde}, with $s=\bar s$, and $\alpha\in\R$, let $w = \alpha v$. We deduce that $w\in H^{\bar s + \ep}(\Omega)$
for $\ep>0$ small, and that
\[
\lambda_1(d,\bar s)\int_\Omega m\psi_{\bar s}w = \alpha \scal{ d^{\bar s} \lapneu{\bar s}v}{\psi_{\bar s}}
=\alpha \scal{ L_{\bar s}\psi_{\bar s}}{\psi_{\bar s}} = 0,
\]
that is $w $ satisfies \eqref{eq:assumpt_curvetta}. Thus Lemma \ref{lem:curvetta} applies, and we denote with
$u(s)$ the corresponding curve. Let us consider the map
$$
f(s):=\scal{ d^{s}\lapneu{s} u(s)}{u(s) },
$$
that, thanks to \eqref{equality:v}, satisfies
$$
f(s)\ge \lambda_1(d,s)\quad \text{and }\quad f(\bar s)=\lambda_1(d,\bar s).
$$
Then it will be enough to show that $\bar s$ is a maximum point
of $f$. By direct computation we obtain
\begin{align*}
f'(s)
=&
\scal{ L_{s}u(s)}{u(s)}
+2\scal{  d^{s}\lapneu{s}\dot u(s)}{ u(s)},
\\
f''(s)
=&
\scal{ T_{s}u(s)}{u(s)}
+4\scal{ L_{s}\dot u(s)}{ u(s)}
+2\scal{  d^{s}\lapneu{s}\dot u(s)}{\dot u(s)}
\\
&+2\scal{  d^{s}\lapneu{s}\ddot u(s)}{ u(s)},
\end{align*}
where $T_{s}$ is defined in \eqref{eq:LeT}. Notice that
\eqref{critical:s} implies that $f'(\bar s)=0$. Recalling \eqref{eq:coeffdiv}, we have
\[
\scal{ T_{\bar s}\psi_{\bar s}}{\psi_{\bar s}} = \sum_k(d\mu_k)^{\bar s} \ln^2(d\mu_k) a_k^2= \scal{ L_{\bar s}v}{\psi_{\bar s}} = \scal{  d^{\bar s}\lapneu{\bar s}v}{v} .
\]
On the other hand, \eqref{equation:ddotu} yields
\[
\scal{d^s \lapneu{s}\ddot u(s)}{ u(s)}|_{s=\bar s} =
\scal{ d^{\bar s} \lapneu{\bar s} \psi_{\bar s}}{\ddot u(\bar s)}
=-\lambda_1(d,\bar s)\into mw^{2}.
\]
Recalling that $w=\alpha v$ and using \eqref{equality:v},
we obtain
\beq\label{last}
f''(\bar s)=\scal{ d^{\bar s} \lapneu{\bar s} v}{v}
\left[1+4\alpha+2\alpha^2
\right]-2\alpha^2\lambda_1(d,\bar s) \into mv^{2}.
\eeq
Now, in case the last integral in \eqref{last} is nonnegative, then choosing  $\alpha=-1$ we obtain $f''(\bar s)<0$, namely $\bar s$ is a maximum point
for $f$ and the result follows. While, in case $\dys\into m v^{2}< 0$, we take into account \eqref{problem:tildev}
and we exploit the definition of $\la_{-1}(1,s)$ in
\eqref{def:la-1} to obtain
\[
\into mv^{2}\geq
\frac{1}{\la_{-1}(1,\bar s)}
\scal{ \lapneu{\bar s}v}{v}.
\]
Then, recalling Lemma \ref{lem:d1d2}, equation \eqref{last} becomes
\[
f''(\bar s)\leq
\scal{d^{\bar s} \lapneu{\bar s} v}{ v}
\left[1+4\alpha+2\alpha^{2}\left(
1- \frac{\la_{1}(1,\bar s)}{\la_{-1}(1,\bar s)}
\right)
\right]<0
\]
when choosing
\[
\alpha=-\left[1-\frac{\lambda_{1}(1,\bar s)}{\lambda_{-1}(1,\bar s)}\right]^{-1}.
\qedhere
\]
\end{proof}

The above theorem, which introduces our crucial abstract hypothesis
\eqref{eq:la1la-1} may not be immediately biologically interpreted;
however, in the next result we furnish a concrete, although more
restrictive, condition which is easier to check and better describes
our results from a biological point of view.
\begin{lemma}\label{lem:easier2check}
If $m\in {\mathcal M}$,  $m\geq -M$ and
\begin{equation}\label{eq:reduced main ass}
\lambda_{1}(m,1,1)<  \frac{\mu_1}{M}
\end{equation}
then
\begin{equation}\label{eq:la-1>la1}
-\lambda_{-1}(m,1,s) > \lambda_{1}(m,1,s)\qquad \text{for every }0<s\le1.
\end{equation}
\end{lemma}
\begin{proof}
Let $\psi_{-1,s}$ be an eigenfunction associated with $\lambda_{-1}(m,1,s)$ as in \eqref{def:la-1}.
Writing $\psi_{-1,s} = h + \tilde \psi_{-1,s}$, with $h\in \R$ and
$ \tilde \psi_{-1,s}$ with zero average,  we have that
\[
\int_\Omega m \psi_{-1,s} = 0
\qquad\implies\qquad
h = -\frac{1}{m_0|\Omega|}\int_\Omega m \tilde\psi_{-1,s},
\]
where $m_0<0$ is given in \eqref{not:muk}. Then
\[
-1 = \int_\Omega m \psi_{-1,s}^2 = -\frac{1}{m_0|\Omega|}\left(\int_\Omega m \tilde\psi_{-1,s}\right)^2
 + \int_\Omega m \tilde \psi_{-1,s}^2
\geq -M\int_\Omega  \tilde \psi_{-1,s}^2.
\]
Recalling the Poincar\'e inequality \eqref{eq:poincar} we have
\[
1 \leq M \left\|\tilde \psi_{-1,s}^2\right\|^2_{L_2}  \leq \frac{M}{\mu_1^s} \scal{ \lapneu{s}\psi_{-1,s}}{\psi_{-1,s}}
= -\lambda_{-1}(m,1,s)\cdot \frac{M}{\mu_1^s}.
\]
By Proposition \ref{prop:lambda1dasopra} and \eqref{eq:reduced main ass} we finally infer
\[
\lambda_{1}(m,1,s)\leq \mu_1^{s-1}\lambda_{1}(m,1,1)< \mu_{1}^{s-1}\frac{\mu_1}{M}
\leq -\lambda_{-1}(m,d,s).
\qedhere
\]
\end{proof}
\begin{proof}[Proof of Theorem \ref{thm:main_intro}]
In view of Theorem \ref{thm:squarepos} and Lemma \ref{lem:easier2check} we have to check that,
if $m\in\tM$ and either \eqref{eq:mainassnew} or \eqref{eq:mainassA} hold, then
\eqref{eq:reduced main ass} follows.
Let $\lambda_1^{\text{Dir}}(B_\rho)=\rho^{-2}\lambda_1^{\text{Dir}}(B_1)$
denote the first Dirichlet eigenvalue of $-\Delta$ in $B_\rho$, with
eigenfunction $\eta\in H^1_0(B_\rho)\subset H^1(\Omega)$ normalized in $L^2(B_\rho)$. Considering
\[
\hat\eta(x) = \eta(x) -\frac{1}{m_0|\Omega|}\int_{B_\rho} m \eta,
\]
we obtain
\begin{align*}
\int_\Omega m\hat\eta^2
& =
 -\frac{1}{m_0|\Omega|}\left(\int_\Omega m \eta\right)^2
 + \int_\Omega m \eta^2
 \geq
 -\frac{\delta^2}{m_0|\Omega|}\left(\int_{B_\rho} \eta\right)^2
 + \delta
\\
&= -\frac{\delta^2}{m_0|\Omega|}\left(\frac{2|\partial B_1|}{\lambda_1^{\text{Dir}}(B_1)}\right)^{1/2}\rho^{N/2} + \delta \geq 0,
\end{align*}
where the last inequality holds because  $m_0<0$,
so that the left hand side is positive. Taking into account the equivalent
expression of $\lambda(m,1,1)$ in terms of the Rayleigh quotient
and using \eqref{eq:mconlapalla}, we infer
\[
\lambda_1(m,1,1) \leq \frac{\dys\int_\Omega|\nabla \hat\eta|^2}{\dys\int_\Omega m\hat\eta^2}
\leq\frac{\rho^{-2}\lambda_1^{\text{Dir}}(B_1)}{\dys-\frac{\delta^2}{m_0|\Omega|}\left(\frac{2|\partial B_1|}{\lambda_1^{\text{Dir}}(B_1)}\right)^{1/2}\rho^{N/2} + \delta}.
\]
Then \eqref{eq:reduced main ass} holds true once
\[
-\frac{1}{m_0|\Omega|}\,\left(\frac{2|\partial B_1|}{\lambda_1^{\text{Dir}}(B_1)}\right)^{1/2}
\frac{\mu_1 \delta^2\rho^{2+N/2}}{M} > \lambda_1^{\text{Dir}}(B_1) - \frac{\mu_1 \delta\rho^2}{M}.
\]
This is true in case \eqref{eq:mainassnew} holds,
since in such case the right hand side is negative; on the other hand, when the right hand side is positive,
the above inequality holds true under \eqref{eq:mainassA}, choosing
$A$ as in \eqref{eq:Aesp}.
\end{proof}
To conclude this section, we turn to discuss the periodic problem.
Let the spectral periodic fractional Laplacian
$(-\Delta_{\mathrm{per}})^{s}$ be defined as in
\eqref{eq;periodic_fraclap}. As we mentioned, the key argument to exploit our previous analysis is that, on periodic functions,
$(-\Delta_{\mathrm{per}})^{s}$ coincides with the
fractional Laplacian on the full space $(-\Delta_{\R^N})^{s}$ (see
\eqref{eq:fraclap_RN}). Actually, this was already observed in \cite{ancora_Stinga}, but for the reader's convenience we detail the argument here.
For easier notation, let us fix $l_1=\dots=l_N=2\pi$, i.e.
$\mathcal{C}_l=(0,2\pi)^N$. Using complex notation, we have that the periodic eigenfunctions
of $-\Delta$ in $(0,2\pi)^N$ are the functions $\varphi_k(x)=e^{-ik \cdot x}$, indexed by $k\in\Z^N$ and corresponding to the eigenvalues $\nu_k=|k|^2$. Let $u$ be $(0,2\pi)^N$-periodic;
up to normalization factors we obtain
\begin{equation}\label{eq:per1}
u(x)= \sum_{k\in\Z^N} u_k e^{-ik \cdot x}, \qquad\text{where }u_k =\int_{\mathcal{C}_l}u(x) e^{-ik \cdot x}\,dx,
\end{equation}
and consequently
\[
(-\Delta_{\mathrm{per}})^{s} u(x)= \sum_{k\in\Z^N} |k|^{2s} u_k e^{-ik \cdot x}.
\]
\begin{proof}[Proof of  Proposition \ref{prop:periodic_is_spectral}]
Let $f\from\R\to\R$ be continuous and
$2\pi$-periodic; then $f$ belongs to the space $\mathcal{S}'(\R)$ of tempered distributions, and it is well known (see e.g. \cite[Ch. II]{MR0166596}) that its Fourier transform is, up to normalization factors,
\[
\hat f (\xi)= \sum_{k\in\Z} f_k \delta_1 (\xi - k), \ \ \xi\in\R,\qquad\text{where }f_k =\int_0^{2\pi}f(x) e^{-ikx}\,dx,
\]
and $\delta_n$ denotes the Dirac delta in $\R^{n}$, $n\geq 1$. Indeed, it suffices to transform
both sides of the identity
\[
f(x)= \sum_{k\in\Z} f_k e^{-ik x},
\]
which holds true in $\mathcal{S}'(\R)$. Recalling that
\[
\delta_N(x_1,\dots,x_N)=\delta_1(x_1)\otimes \dots \otimes\delta_1(x_N),
\]
it is not difficult to generalize the above formula to the $N$-dimensional setting,
obtaining that, if $u$ is $(0,2\pi)^N$-periodic, then
\[
\hat u (\xi)= \sum_{k\in\Z^N} u_k \delta_N (\xi - k), \ \ \xi\in\R^N,
\]
and $u_{k}$ are given in \eqref{eq:per1}.
Exploiting \eqref{eq:fraclap_RN}, \eqref{eq;periodic_fraclap} and recalling that
$\nu_{k}=|k|^{2}$,  we obtain:
\[
\widehat{(-\Delta_{\R^N})^s u} (\xi) = |\xi|^{2s} \widehat{u}(\xi) = \sum_{k\in\Z^N} |\xi|^{2s} u_k \delta_N (\xi - k)
= \sum_{k\in\Z^N} |k|^{2s} u_k \delta_N (\xi - k) = \widehat{(-\Delta_{\mathrm{per}})^s u} (\xi),
\]
and the desired result follows.
\end{proof}
Once the equivalence between $(-\Delta_{\mathrm{per}})^{s}$ and $(-
\Delta_{\R^N})^{s}$ is established, one can easily repeat the
arguments introduced for the Neumann case, because
$(-\Delta_{\mathrm{per}})^{s}$ is a spectral operator as pointed out
in \eqref{eq;periodic_fraclap}: since $\nu_0=0$ and
$\nu_k\to+\infty$ as $|k|\to \infty$, these arguments are exactly the
same, except for the use of the regularity results from \cite{cast} (see Appendix
\ref{app:eig}): following the arguments in \cite{cast},
these results can be proved also for $(-\Delta_{\mathrm{per}})^{s}$,
even though in this case it is much easier to use the regularity theory
for  $(-\Delta_{\R^N})^{s}$, which is well established in
\cite{cabresire}.
Hence,  the whole argument of the Neumann case can be followed
also in the periodic case, without any change, but replacing
$(\mu_k,\phi_k)_{k\in\Z}$
with $(\nu_k,\varphi_k)_{k\in\Z^N}$. In particular, the proof of Theorem \ref{thm:main_intro_per} follows as well.

\section{Optimization on \texorpdfstring{$m$}{m}}\label{sec:opt_m}

In this section we will briefly analyse  the optimization of
$\lambda_{1}(m,d,s)$ with respect to $m$.
In this analysis it is convenient to fix $\underline{m},\,
\overline{m}\in \R^{+}$, $m_{0}\in (-\underline{m},0)$ and
take $m$ in the following class.
\begin{equation}\label{ipommm}
m\in\overline{{\mathcal M}}:=\left\{-\underline{m}\leq m(x)\leq \overline{m},
\,\into m=m_{0}|\Omega|,
  \,  m^+ \not\equiv 0\right\}.
\end{equation}
\begin{remark}\label{rem:51}
When $m$ satisfies \eqref{ipommm}, condition \eqref{basso}
can be rewritten as
\[
\la_{1}(m,d,s)\geq \frac{d^{s}\mu_{1}^{s}|m_{0}|}{\overline{m}|m_{0}|
+[\max(\overline{m},\underline{m})]^{2}}.
\]
\end{remark}
The following result is proved in \cite{lilo,dego,loya}.
\begin{lemma}\label{bathtub}
Let $f\in L^{1}(\Omega) $.
Then the maximization problem
$\dys
\sup_{m\in\overline{{\mathcal M}}}\int_\Omega fm
$
is solved by
\beq\label{eq:bang}
m=\overline{m}\chi_{D}-\underline{m}\chi_{D^{c}},
\eeq
for some subsets $D\subset \Omega$, $D^{c}=\Omega\setminus D$, such that
\beq\label{eq:D}
|D|=|\Omega|\frac{\underline{m}+m_{0}}{\underline{m}+\overline{m}}.
\eeq
\end{lemma}
\begin{theorem}\label{th:opm}
For every $d>0$ and
$s$ fixed, there exists $\underline{\la}_{1}(d,s)$
solution of the minimization problem
\beq\label{infm}
\underline{\la}_{1}(d,s)=\inf_{\overline{{\mathcal M}}}\la_{1}(m,d,s).
\eeq
Moreover, $\underline{\la}_{1}(d,s)$ is achieved by
$m=\overline{m}\chi_{D}-\underline{m}\chi_{D^{c}}$, for some $D\subset \Omega$, independent of $d$,
which satisfies \eqref{eq:D}.
\end{theorem}
\begin{proof}
Notice that, for every $m\in \overline{{\mathcal M}}$, Lemma \ref{bathtub} implies
\begin{align*}
\la_{1}(m,d,s)=d^{s}\frac{\scal{(-\Delta)^{s}\psi_{1}}{\psi_{1}
}}{\into m\psi_{1}^{2}}
&\geq
d^{s}\frac{\scal{(-\Delta)^{s}\psi_{1}}{\psi_{1}
}}{\into \left(\overline{m}\chi_{D}-\underline{m}\chi_{D^{c}}\right)
\psi_{1}^{2}}
\\
&\geq\la_{1}(\overline{m}\chi_{D}-\underline{m}\chi_{D^{c}},d,s),
\end{align*}
for $D$ satisfying \eqref{eq:D}. Since $\overline{m}\chi_{D}-\underline{m}\chi_{D^{c}}\in \overline{{\mathcal M}}$ the conclusion
follows.
\end{proof}
{\bf Biological Interpretation.} Theorem
\ref{th:opm} gives a partial answer to the question of
the best way to model an environment to favor survival.
As observed in \cite{loya, MR1014659}, minimizing $\lambda_{1}$
with respect to $m$ is equivalent to determine the optimal spatial
arrangement of the favorable and unfavorable parts of the habitat
for the species to survive. In Theorem \ref{th:opm} we observe that,
also in the case of long jumps diffusion, the optimal habitat is of bang-bang type
(even though the actual shape of the favorable set $D$ may depend on the diffusion
strategy adopted by the population).
\begin{remark}
The above result holds true also in the periodic setting;
we refer the interested reader to \cite{MR2364810}, where the relation
between the Neumann case and the periodic one has been analyzed in details, for $s=1$.
\end{remark}

\section{Discussion}\label{sec:final}

Reaction-diffusion equations are a well established mathematical model for the description 
of population dispersal \cite{MR1952568}. In the model \eqref{eq:diff_eq_neu} that we analyze, three main 
features are present: 
\begin{itemize}
 \item the population moves according to a random walk either of Brownian type ($s=1$), or allowing for long jumps ($s<1$);
 \item the boundary of the domain acts as a ``reflecting barrier'' (Neumann boundary conditions);
 \item the motility coefficient depends on the fractional power $s$.
\end{itemize}
In the context of fractional models, boundary conditions have to be incorporated in the operator, and 
different interpretations of no flux boundary conditions are available \cite{MR3217703}. Our model 
deals with 
a reflecting boundary, but we can treat the fractional Laplacian on the
whole $\R^N$ in periodic environments as well. Moreover, although the choice of a motility 
coefficient depending on $s$ is motivated by dimensional arguments \cite{MeKa}, we can also treat 
constant diffusivities. 

It is well known \cite{beroro} that the principal positive eigenvalue $\lambda_{1}(m,d,s)$, 
associated to the linearized, stationary version of \eqref{eq:diff_eq_neu}, acts as a threshold value
for the persistence of the population: solutions to \eqref{eq:diff_eq_neu} survive (i.e. they tend to 
the unique positive steady state, as $t\to+\infty$) if and only if $\lambda_1(m,d,s) < 1$, while in the complementary case they extinguish (i.e. they tend to $0$). Therefore minimizing 
$\lambda_1(m,d,s)$, with respect to the parameters involved, amounts to boost the survival chances of the species.

Our optimization analysis deals with the following two problems:
\begin{enumerate}
\item
{\bf Optimal strategy problem:}
given a population and its habitat, is the Brownian hunting strategy more effective than the fractional one, in order to survive?
\item{\bf Optimal design problem for wildlife refugees:}
given a population, what is the best way  to model
 the environment in order to boost survival?
\end{enumerate}


Our main results are concerned with question (1).
In particular, in Theorem  \ref{prop:dep_on_d}
we analyze the case of  very large/small environments:  
in sizeable environments the local diffusion is more successful, while in very 
small ones a fractional diffusion strategy would be preferable. This 
result suggests the presence of a transition situation, which is analyzed 
in Theorems \ref{thm:main_intro}, \ref{thm:main_intro_per}. 
Here we show that, assuming that the habitat is either not
too fragmented or not too hostile in average, the optimal diffusion strategy is always the Brownian one, or the
fractional with the smallest admissible exponent, depending on the relative sizes of the spatial domain and of 
the motility.
This phenomenon actually relies on an abstract comparison
condition involving the eigenvalues $\lambda_{1}$ and
$\lambda_{-1}$ (equation \eqref{eq:la-1>la1}).
It is then natural to wonder whether or not this condition is merely
technical; equivalently, are  there habitats in which the optimal
strategy is attained by a fractional diffusion with exponent
$\overline{s}\in (0,1)$?
We provide some numerical simulations
that suggest that this can actually occur, when assumption \eqref{eq:la-1>la1} is violated.

In the square $\Omega = (0,\pi)\times(0,\pi) \subset \R^2$ we consider the two environments
\begin{equation}\label{eq:model_env}
m_{1}(x_1,x_2) :=
\begin{cases}
8  & \,x_1^2+x_2^2 <  1\\
 -1  & \,x_1^2+x_2^2 >  1,\\
\end{cases}
\qquad
m_{2}(x_1,x_2) :=
\begin{cases}
1  & \,x_1^2+x_2^2 <  1\\
 -1  & \,x_1^2+x_2^2 >  1.\\
\end{cases}
\end{equation}
Notice that, with both choices,
$m\in\tM$. For the two possibilities, the eigenvalues $\lambda_{1}(m,1,s)$,
$\lambda_{-1}(m,1,s)$ are numerically evaluated, by truncating
the Fourier series, for
$s\in\{i/100:i=1,\dots,100\}$.
As shown in Figure \ref{fig:cond}, condition  \eqref{eq:la-1>la1}
appears to be satisfied for $m=m_{1}$, whereas
it does not hold when $m=m_{2}$.
However, notice also that in this case
$\mu_1=\mu_2=1$ (achieved by $\phi_1(x)=\cos x_1$,
$\phi_2(x)=\cos x_2$), so that we are in the situation
described by the first conclusion of  Theorem \ref{prop:dep_on_d},
namely,  all the graphs are increasing in $s$.
Then in this case the minimum of $\lambda_{1}(m,1,s)$
is achieved in $\lambda_{1}(m,1,0^{+})$, no matter
whether or not condition  \eqref{eq:la-1>la1} is verified.
\begin{center}
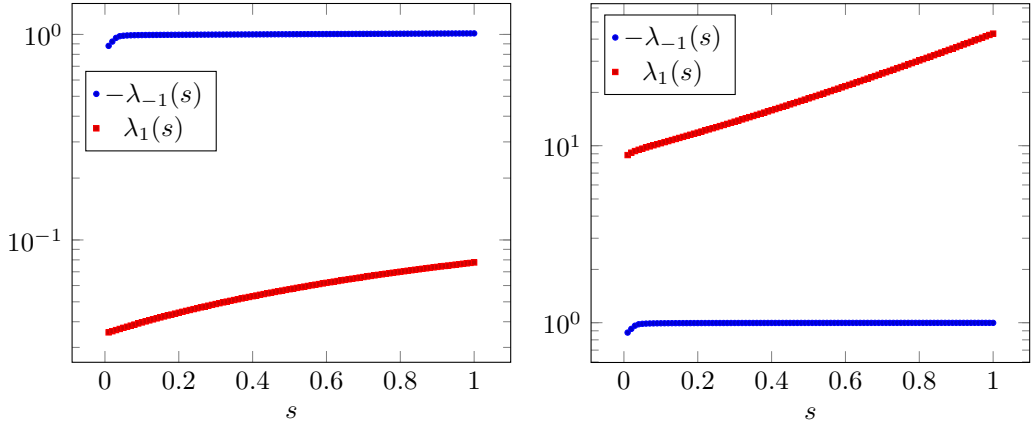
\begin{figure}
\begin{tikzpicture}[baseline]
\pgfplotsset{every axis/.append style={width=0.5\linewidth}}
\begin{axis}[
     name=true,
    ymode=log,
    xlabel=$s$,
    legend style={at={(0.03,0.7)},anchor=west},
    legend entries={$-\lambda_{-1}(s)$,$\lambda_{1}(s)$}
    ]
\addplot+[
    only marks,
    mark size=1pt,
    filter discard warning=false,
    ]
table[x={s},y expr={-\thisrow{lntrue}}]{data.dat};
\addplot+[
    only marks,
    mark size=1pt,
    filter discard warning=false,
    ]
table[x={s},y={lptrue}]{data.dat};
\end{axis}
\begin{axis}[
   at={(true.outer north east)},anchor=outer north west,
    xshift=3mm,
    name=false,
    ymode=log,
    xlabel=$s$,
    legend pos=north west,
    legend entries={$-\lambda_{-1}(s)$,$\lambda_{1}(s)$}
    ]
\addplot+[
    only marks,
    mark size=1pt,
    filter discard warning=false,
    ]
table[x={s},y expr={-\thisrow{lnfalse}}]{data.dat};
\addplot+[
    only marks,
    mark size=1pt,
    filter discard warning=false,
    ]
table[x={s},y={lpfalse}]{data.dat};
\end{axis}
\end{tikzpicture}%
\caption{Testing condition \protect\eqref{eq:la-1>la1} for the model environment
\protect\eqref{eq:model_env}
with $m=m_{1}$ (on the left) and $m=m_{2}$ (on the right).}
\label{fig:cond}
\end{figure}
\end{center}
In Figure \ref{fig:result},
$\lambda_{1}(m,d,s)=d^s\lambda_{1}(m,1,s)$ is plotted
for different choices of the motility coefficient $d<1$.
When $m=m_{1}$ (so that condition \eqref{eq:la-1>la1} is
satisfied) it is possible to observe the transition
of the behavior of $\lambda_{1}(m,1,s)$ from
decreasing to  increasing (while $d$ increases)
developing in the meanwhile a critical point of maximum type.
But, when $m=m_{2}$, in which case condition \eqref{eq:la-1>la1}
is violated, $\lambda_{1}(m,1,s)$ develops also critical points of minimum type, while moving from  decreasing to increasing. Figure \ref{fig:result} suggests a further line of investigation.
Indeed, it would be interesting to find the proper conditions under
which $\lambda_{1}(m,d,s)$ produces internal minimum points,
because this would lead to an optimal diffusion strategy of fractional 
type, but not close to zero.
\begin{center}
\begin{figure}
\begin{tikzpicture}[baseline]
\pgfplotsset{every axis/.append style={width=0.5\linewidth}}
\begin{axis}[
     name=true,
   legend pos=north west,
    legend columns=2,
    xlabel=$s$,
    ylabel={$\lambda_1(m,d,s)$},
    legend entries={$d=0.2$,$d=0.4$,$d=0.6$,$d=0.8$}
    ]
\foreach \c in {.2,.4,...,.8}
{
\addplot+[
    only marks,
    mark size=1pt,
    filter discard warning=false,
    ]
table[x index=0,y expr={(\c)^(\thisrow{s})*\thisrow{lptrue}}] {data.dat};
}
\end{axis}
\begin{axis}[
   at={(true.outer south east)},anchor=outer south west,
    xshift=3mm,
    name=false,
   legend pos=north west,
    legend columns=2,
    xlabel=$s$,
    legend entries={$d=0.16$,$d=0.18$,$d=0.20$,$d=0.22$,$d=0.24$}
    ]
\foreach \c in {.16,.18,...,.24}
{
\addplot+[
    only marks,
    mark size=1pt,
    filter discard warning=false,
    ]
table[x index=0,y expr={(\c)^(\thisrow{s})*\thisrow{lpfalse}}] {data.dat};
}
\end{axis}
\end{tikzpicture}%
\caption{The map $s\mapsto\lambda_1(m,d,s)$, for several values of $d
$, with $m=m_{1}$ (on the left) and $m=m_{2}$ (on the right)
both given in \protect\eqref{eq:model_env}.}
\label{fig:result}
\end{figure}
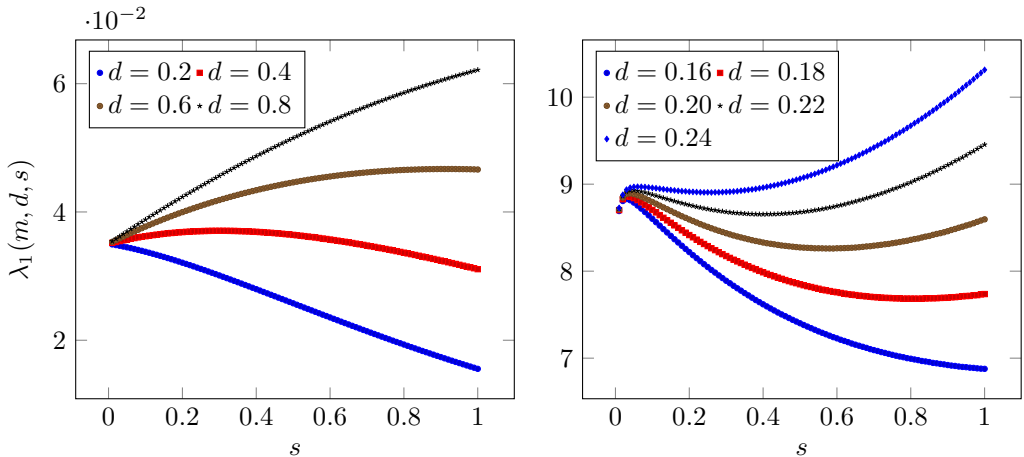
\end{center}

Question (2) is discussed in Section \ref{sec:opt_m},
where we show that the optimal habitat is of bang-bang type
also in the case of long jumps diffusion (for standard diffusions this is already well known). 
However, the actual shape of the favorable set $D$ may depend on the 
diffusion strategy adopted by the population. As a consequence,
it is natural to investigate the further qualitative properties
of the favorable region $D$, in particular, to wonder  
whether or not $D$ is connected,
as this is related to  the detection of possible fragmentation of the
optimal environment.
The connectedness of $D$ has been  obtained in the local diffusion case, for
$N=1$, in  \cite{MR1105497,loya,dego}. This line of research has been pursued  in higher dimension in \cite{MR2364810}, where a sharp analysis of the optimal environment is performed in the standard diffusion case $s=1$. In particular, when $\Omega$ is a bi-dimensional
rectangle, by combining monotone Steiner rearrangements and numerical simulations it appears that $D$ and $\Omega\setminus D$ can be of two main types: ball-shaped or stripe-shaped.
In addition, when the ratio $|D|/|\Omega\setminus D|$ is sufficiently small, it was conjectured that $D$ should be a quarter of circle, centered in one of the corners of $\Omega$. Nonetheless,
such conjecture has been recently disproved in \cite{Privat}.
By using symmetrization arguments on the extension problem (Remark \ref{rem:extension}), we expect that part of such analysis may be carried also to the case $s<1$, even though this falls
beyond the scope of the present paper.


\appendix \section{Existence and Properties of \texorpdfstring{$\lambda_{1}(m,d,s)$}{the first eigenvalue}. }\label{app:eig}

In this appendix we prove the existence of two sequences of eigenvalues of the operator
$\lapneu{s}$ and we show  the  main properties of
$\lambda_{1}(m,d,s)$ and of its associated eigenfunction.

Taking into account the $L^2$-spectral decomposition (see \eqref{not:muk}), we consider the functional space
\begin{equation}\label{eq:defHs}
H^s(\Omega) = \left\{u = u_0 + \sum_{k=1}^{\infty} u_k \phi_k
\in L^2(\Omega)\, : \,
\sum_{k=1}^{\infty} \mu_k^s u_k^2<+\infty\right\}.
\end{equation}
Note that, as shown in \cite[Lemma 7.1]{cast}, this definition of $H^s(\Omega)$ is equivalent to
the usual one given in terms of the Gagliardo semi-norm
\[
[u]^2_{2,s} = \int_\Omega \int_\Omega \frac{|u(x)-u(y)|^2}{|x-y|^{N+2s}}\,dxdy.
\]
In $H^s(\Omega)$ it is well-defined the fractional differential operator
\begin{equation}\label{eq:def_frac_lap}
\lapneu{s} u = \sum_{k=1}^{\infty} \mu_k^s u_k\phi_k,
\end{equation}
and, taking into account \eqref{not:muk},
the norm in $H^s(\Omega)$ can be written as
\begin{equation}\label{eq:def_oper}
\| u \|_{H^s}^2 = u_0^2 + \sum_{k=1}^{\infty} \mu_k^s u_k^2 =
\left(\frac{1}{|\Omega|}\int_\Omega u \right)^2 + \scal{ \lapneu{s}u}{u}.
\end{equation}
In the following, we will prove the existence of a double sequence of
eigenvalues for problem \eqref{eq:eigen_u}, and some qualitative properties
of the eigenfunctions.
\begin{proposition}\label{prop:lambdak}
Problem \eqref{eq:eigen_u} admits two unbounded sequences
of eigenvalues:
\[
\dots \leq \lambda_{-2} \leq \lambda_{-1} < \lambda_0 = 0 <\lambda_1 < \lambda_2 \leq \lambda_3 \leq \dots
\]
Furthermore, both the eigenvalues and the
(normalized) eigenfunctions depend continuously on $m$.
In particular, \eqref{eq:def_eig} holds.
\end{proposition}
\begin{proof}
The results for $s=1$ are standard, so we restrict to the case $0<s<1$.
First of all, the simple eigenvalue $\lambda_0=0$, with constant
eigenfunction, can be  computed directly. The other
eigenvalues can be obtained by restricting to the space
\[
V := \left\{u\in H^s(\Omega) \,:\, \scal{u}{1} = 0\right\}.
\]
Indeed, in this space we can use the equivalent scalar product
\[
\scal{u}{v}_V = \sum_{k=1}^{\infty} \mu_k^s u_k v_k
\]
and we have that the linear operator $T: V \to V$ defined by
\[
\scal{Tu}{v}_V = \int_\Omega muv
\]
is symmetric and compact, thanks to the compact embedding of
$H^s(\Omega)$ in $L^2(\Omega)$ (recall the definition of $H^s(\Omega)$ in \eqref{eq:defHs}).
As a consequence, we can apply standard results in spectral theory of self-adjoint compact
operators to obtain the existence and the variational characterization of the eigenvalues
(see e.g. \cite[Propositions 1.3, 1.10]{deFig}), as well as the continuity
property of the spectrum (see the book by Kato \cite{Kato}).
\end{proof}
\begin{remark}\label{rem:extension}
Alternatively, following \cite{cabresire,cast}, the above result can be obtained by means of an
extension problem in
${\mathcal C}:=\Omega\times(0,\infty)$.
Indeed, let $1-2s=:a\in(-1,1)$ and
$$
{\mathcal H}^{1;a}({\mathcal C}):=\left\{v=v_0+\tilde{v}:v_0\in\R,\,\int_{\mathcal C} y^a\left(|\nabla \tilde v|^2+\tilde{v}^2\right)\,dxdy<+\infty\right\}.
$$
It is known (see \cite{nek}) that, for $\partial\Omega$ sufficiently smooth, the elements of
$H^s(\Omega)$
coincide with the traces of functions in ${\mathcal H}^{1;a}({\mathcal C})$. As a consequence,
any $u\in H^s(\Omega)$ admits a unique extension
$v\in {\mathcal H}^{1;a}({\mathcal C})$ which
achieves
\begin{equation}\label{eq:min_extens}
\min\left\{\int_{\mathcal C} y^a|\nabla v|^2\,dxdy : v(x,0) = u(x)\right\}.
\end{equation}
Then \eqref{eq:eigen_u} is equivalent to
$$
\begin{cases}
\diverg(y^{a}\nabla v)=0 & \text{in }\mathcal C\\
\partial_\nu v = 0 & \text{on }\partial\Omega \times (0,\infty)\\
{D(s)}{d^{s}}
\partial_\nu^{a} v(x,0)= \lambda m(x)v(x,0)& \text{in }\Omega,
\end{cases}
$$
where the structural constant $D(s)$ is known to be
\[
{D(s)}=2^{2s-1}\frac{\Gamma(s)}{\Gamma(1-s)},
\]
so that one has the following characterization:
\begin{align}\label{chara}
\lambda_{1}(m,d,s)
=
d^{s}D(s)\min_{{\mathcal H}^{1;a}({\mathcal C})}\left\{\int_{\mathcal C}y^{a}|\nabla v|^{2}dxdy :
\into mv^{2}(x,0)dx=1
\right\}.
\end{align}
Note that the last formulation can be rewritten in terms of a suitable Rayleigh quotient.
\end{remark}

\begin{proposition}\label{prop:regularity}
Let $\psi$ be any eigenfunction of problem $\eqref{eq:eigen_u}$. Then
\[
\psi\in H^{2s}(\Omega).
\]
Furthermore, $\psi\in C^{0,\alpha}(\Omega)$ for every $\alpha < 2s$,
whenever $s\leq 1/2$, and $\psi\in C^{1,\alpha}(\Omega)$ for every $\alpha < 2s-1$,
in case $s>1/2$.
\end{proposition}
\begin{proof}
Since $m\in L^\infty(\Omega)$, equation \eqref{eq:eigen_u} implies that $\lapneu{s}
u \in L^2(\Omega)$, that is, recalling \eqref{not:muk}
\[
\sum_k (\mu_k^{s} u_k)^2 < +\infty;
\]
the Sobolev regularity follows by the definition of $H^{2s}(\Omega)$ given in \eqref{eq:defHs}. On the other hand, the H\"older regularity of the
eigenfunctions is a consequence of the regularity theory developed by
Caffarelli and Stinga \cite[Theorem 1.5]{cast}, and of a standard
bootstrap argument.
\end{proof}
Thanks to Proposition \ref{prop:lambdak} we have that there exists
a solution to the linear problem
\begin{equation}\label{eq:eigen_psi}
d^s \lapneu{s} \psi= \lambda_1(m,d,s) m \psi \qquad x\in\Omega,
\end{equation}
and now we  turn to the study of the properties of the first
eigenvalue $\lambda_1$ and the associated eigenfunction
$\psi$.
First of all, in order to show that  $\lambda_1$ is simple,
we will exploit the following lemma, which concerns a convexity property
of the $H^s(\Omega)$ semi-norm.
\begin{lemma}\label{lem:parti_positive}
Let $u\in H^s(\Omega)$, $0<s<1$. Then $u^\pm \in H^s(\Omega)$ and
\begin{equation}\label{eq:convex_seminorm}
 \scal{ \lapneu{s}u}{u} \geq
 \scal{ \lapneu{s}u^+}{u^+} + \scal{ \lapneu{s}u^-}{u^-},
\end{equation}
and the strict inequality holds whenever $u^\pm$ are both nontrivial.
\end{lemma}
\begin{remark}
The lemma enlightens a substantial difference between the nonlocal and the local case. Indeed, when $s=1$, the equality sign in \eqref{eq:convex_seminorm} always holds for any
$u$.  A similar result in the periodic case has been shown in
\cite[Proposition 3.1]{beroro}.
\end{remark}
\begin{proof}[Proof of Lemma \ref{lem:parti_positive}]
Let $v\in {\mathcal H}^{1;a}({\mathcal C})$ be the extension of $u$
given in \eqref{eq:min_extens}. Then
$v^\pm\in {\mathcal H}^{1;a}({\mathcal C})$ and, taking into account Remark \ref{rem:extension},
their traces $u^\pm$ belong to $H^s(\Omega)$. Therefore
\[
\begin{split}
\scal{ \lapneu{s}u}{u} = \int_{\mathcal C} y^a|\nabla v|^2\,dxdy &=
\int_{\mathcal C} y^a|\nabla v^+|^2\,dxdy  + \int_{\mathcal C} y^a|\nabla v^-|^2\,dxdy \\
&\geq \int_{\mathcal C} y^a|\nabla w_+|^2\,dxdy + \int_{\mathcal C} y^a|\nabla w_-|^2\,dxdy \\
&=\scal{ \lapneu{s}u^+}{u^+} + \scal{ \lapneu{s}u^-}{u^-},
\end{split}
\]
where $w_\pm$ solve the minimization problem \eqref{eq:min_extens} with traces $u^\pm$.

Finally, if $v^\pm$ are both nontrivial, the strong maximum principle
\cite[Remark 4.2.]{cabresire} implies that they cannot solve \eqref{eq:min_extens},
thus the strict inequality holds.
\end{proof}
\begin{proposition}\label{prop:lambda1}
The eigenvalue $\lambda_1(m,d,s)$ is simple, and the associated eigenfunction does not change sign.
Moreover the map
\[
\mathcal{M}\times\R^+\times(0,1]\ni (m,d,s) \mapsto \lambda_1
\]
is analytic.
\end{proposition}
\begin{proof}
The fact that $\lambda_1$ is simple, with one-signed eigenfunction,
can be deduced arguing as in \cite[Theorem 1.13]{deFig}, taking into account
Lemma \ref{lem:parti_positive}.

To prove the second part of the statement, let $(m^*,d^*,s^*)\in \mathcal{M}\times\R^+
\times(0,1]$ be fixed and $\lambda_1,\psi_1$ denote the first eigenvalue and
the non-negative, normalized first eigenfunction for the corresponding problem \eqref{eq:eigen_psi} with weight $m^*$, coefficient $d^*$ and exponent $s^*$.
If $\sigma>0$ is sufficiently
small we have that, by Propositions \ref{prop:lambdak} and \ref{prop:regularity}, the map
\[
\begin{array}{c}
\dys  \mathcal{F}\from \mathcal{M} \times \R^+ \times (s^*-\sigma,s^* + \sigma) \times H^{2(s^*-\sigma)}(\Omega)
\times \R^+ \to  L^2 (\Omega) \times \R,
\smallskip\\
\dys \mathcal{F}(m,d,s,u,\la)=\left(d^s\lapneu{s}u-\la m u,\into m u^{2}-1\right),
\end{array}
\]
is well defined, and
\[
\mathcal{F}(m^*,d^*,s^*,\psi_1,\la_1)=(0,0).
\]
In order to reach the conclusion, we are going to apply the Implicit Function Theorem to
$\mathcal{F}$, expressing the pair $(u,\lambda)$ as function of $(m,d,s)$. To this aim, computing the derivative one obtains
\[
\partial_{(u,\la)}\mathcal{F}(m^*,d^*,s^*,\psi_1,\la_1)[v,l]
=
\left(
\begin{array}{c}
\dys (d^*)^{s^*}\lapneu{s^*}v-\la_1 m^* v-l m^* \psi_1\smallskip\\
\dys 2\into m^* \psi_1 v
\end{array}
\right).
\]
By Fredholm's Alternative, it suffices to show that the linear
operator above is injective and this is a straightforward consequence of
\eqref{eq:def_eig} and of the fact that $\lambda_1$ is simple.
\end{proof}
As a direct consequence of Proposition \ref{prop:lambda1}
we have the following result.
\begin{corollary}
Let us denote with $\psi_1$ the non-negative, normalized eigenfunction corresponding to
$\lambda_1$. For every $0<s_1<s_2\leq 1$, the map
\[
F \from  \mathcal{M} \times \R^+ \times (s_1 , s_2] \to \R^+ \times H^{2s_1} (\Omega),
\qquad
F(m,d,s)=(\lambda_1, \psi_1)
\]
is analytic.
\end{corollary}
\begin{remark}
It is natural to wonder whether the eigenfunction corresponding to
$\lambda_1$ can be chosen to be strictly positive on $\overline{\Omega}$.
To obtain this result, one may invoke the strong maximum principle
\cite[Remark 4.2, Proposition 4.11]{cabresire}.
This requires more regularity, and the proof can be completed in case
$m\in C^{0,\alpha}(\overline{\Omega})$ (and $\partial\Omega$ is smooth) by
using \cite[Theorem 1.4]{cast}.
For a general $m\in L^\infty(\Omega)$, we can only deduce that the eigenfunction can
not vanish on a set with non-empty interior, but we can not exclude vanishing points, in
particular when $s$ is small.
\end{remark}


\section{Proofs of some auxiliary results.}\label{app:aux}
In this appendix we collect the proofs of Lemma
\ref{lem:curvetta} and \ref{lem:vtilde}.
\begin{proof}[Proof of Lemma \ref{lem:curvetta}]
Let us define $\gamma(t):=\psi_{\bar s}+(t-\bar s)w$, then
\[
u(t):=\dfrac{\gamma(t)}{\sqrt{\into m\gamma^{2}(t)}}
\]
satisfies all the requested properties.

Indeed, $u\in H^{\bar s+ \ep}(\Omega)$ by Proposition \ref{prop:regularity}, and \eqref{eq:psi_1} yields the first equality
in \eqref{equality:v}, while the third one holds by the definition of $u$.
Moreover, \eqref{eq:assumpt_curvetta} implies
\[
\dot u(\bar s)=
\frac{\dot\gamma(\bar s)}{\sqrt{\into m\gamma^{2}(\bar s)}}
-\frac{\gamma(\bar s)}{\left[\into m\gamma^{2}(\bar s)\right]^{3/2}}
\into m\gamma(\bar s)\dot\gamma(\bar s)
=w-\psi_{\bar s}\into m\psi_{\bar s}w=w.
\]
In addition, differentiating twice the last
equality in \eqref{equality:v} we have
\[
\into m\dot u(t)^{2}+\into mu(t)\ddot u(t)=0 \quad \forall t\in (\bar s-\ep,\bar s+ \ep),
\]
so that
$$
\into m\psi_{\bar s}\ddot u(\bar s)=-\into mw^{2}.
$$
This equality implies \eqref{equation:ddotu}
when taking as test function $\ddot u(\bar s)$ in \eqref{eq:psi_1}.
\end{proof}
\begin{proof}[Proof of Lemma \ref{lem:vtilde}]
Exploiting \eqref{def:lapneu}, the equation in \eqref{problem:tildev} rewrites as
\[
\sum_{k=1}^{\infty}(d\mu_{k})^s v_{k}\phi_k = \sum_{k=1}^{\infty}(d\mu_{k})^s \ln(d\mu_{k})a_{k}\phi_k,
\]
where $v_k$ and $a_k$ are the Fourier coefficients of $v$ and $\psi_s$ respectively.
Such problem is solved in $L^2(\Omega)$ by
\begin{equation}\label{eq:coeffdiv}
v_k= \ln(d\mu_{k})a_{k}, \qquad k\geq1,
\end{equation}
and any $v_0\in\R$. Moreover, Proposition \ref{prop:regularity} yields
\[
\sum_{k=1}^{\infty}(d\mu_{k})^{s'} v_{k}^2 =  \sum_{k=1}^{\infty}(d\mu_{k})^{s'} \ln^2(d\mu_{k}) a_{k}^2 \leq
C +  \sum_{k=1}^{\infty}(d\mu_{k})^{2s} a_{k}^2 < +\infty\,,
\]
so that $v\in H^{s'}(\Omega)$, for every $v_0\in \R$.
Finally, recalling \eqref{not:muk} and hypothesis
\eqref{ipom}, we have that
\[
\max_{v_0\in\R} \int_\Omega m(v_0+\tilde v)^2
\]
is uniquely achieved by
\[
v_0=-\dfrac{1}{m_0|\Omega|}\int_\Omega m \tilde v ,
\]
which also satisfies the second condition in \eqref{problem:tildev}.
\end{proof}


\end{document}